\documentclass[preprint,12pt]{elsarticle}

\usepackage{amssymb}
\usepackage{amsthm}
\usepackage{amsmath}

\newtheorem{definition}{Definition}
\newtheorem{theorem}{Theorem}
\newtheorem{lemma}{Lemma}

\journal{Computational Geometry, Theory and Applications}

\newcommand{\IR}{\mathbb{R}}

\begin{document}

\begin{frontmatter}

\title{Optimally solving a transportation problem using Voronoi diagrams\tnoteref{EuroGiga-Voronoi}}

\author[uniBonn]{Darius Gei\ss{}}\ead{wdarius@gmx.de}
\author[uniBonn]{Rolf Klein}\ead{rolf.klein@uni-bonn.de}
\author[uniBonn]{Rainer Penninger\corref{cor1}}\ead{penninge@cs.uni-bonn.de}
\cortext[cor1]{Corresponding author}
\author[uniBerlin]{G{\"u}nter Rote}\ead{rote@inf.fu-berlin.de}

\address[uniBonn]{Rheinische Friedrich-Wilhelms Universit{\"a}t Bonn, Institute of Computer Science I, Friedrich-Ebert-Allee 144, D-53113 Bonn, Germany.}
\address[uniBerlin]{Freie Universit{\"a}t Berlin, Institut f{\"u}r Informatik, Takustra{\ss}e 9, D-14195 Berlin, Germany.}

\tnotetext[EuroGiga-Voronoi]{This work was supported by the European Science Foundation (ESF) in the EUROCORES collaborative research project EuroGIGA/VORONOI.
The results were presented, in preliminary form,
at the
28th European Workshop on Computational Geometry (EuroCG'12), in Assisi, Italy, March 2012~\cite{EUROCG} and at the 18th Annual International Computing and Combinatorics Conference (COCOON 2012), Sidney, August 2012~\cite{COCOON}.
}

\begin{abstract}
 We consider the following variant of the well-known Monge-Kantorovich transportation problem.
Let $S$ be a set of $n$ point sites in $\IR^d$. A bounded set $C \subset \IR^d$ is to be distributed among
the sites $p \in S$ such that (i), each $p$ receives a subset $C_p$ of prescribed volume and (ii), 
the average distance of all points $z$ of $C$ from their respective sites $p$ is minimized.
In our model, volume is quantified by a measure $\mu$, and the distance between a site $p$ and a
point~$z$ is given by a function $d_p(z)$.
Under quite liberal technical assumptions on $C$ and on the functions $d_p(\cdot)$ we show that a solution 
of minimum total cost can be obtained by intersecting with~$C$ the Voronoi diagram of the sites in $S$, 
based on the functions $d_p(\cdot)$ equipped with suitable additive weights.
Moreover, this optimum partition is unique, up to sets 
 of measure zero.
Unlike the deep analytic methods of classical transportation theory, our proof is based directly on
 geometric arguments.
\end{abstract}

\begin{keyword}
Monge-Kantorovich transportation problem \sep earth mover's distance \sep Voronoi diagram with additive weights \sep Wasserstein metric

\end{keyword}

\end{frontmatter}



\section{Introduction}     \label{intro-sec}
In 1781, Gaspard Monge~\cite{m-mtdr-81} raised the following problem. 
Given two sets $C$ and $S$ of equal mass
in $\IR^d$, transport each mass unit of $C$ to a mass unit of $S$
at minimal cost. More formally, given two measures $\mu$ and $\nu$,
find a map $f$ satisfying $\mu(f^{-1}(\cdot)) = \nu(\cdot)$
that minimizes
\[
    \int \! d(z,f(z)) \, \mathrm{d} \mu(z),
\]
where $d(z,z')$ describes the cost of moving $z$ to $z'$.

Because of its obvious relevance to economics, and perhaps due to its
mathematical challenge, this problem has received  a lot of attention. 
Even with the Euclidean distance as cost function~$d$ it is not at all clear
in which cases an optimal map $f$ exists.
Progress by Appell~\cite{a-mdrsc-87} was honored with a prize by the 
Academy of Paris in 1887. Kantorovich~\cite{k-pm-48} achieved a breakthrough
by solving a relaxed version of Monge's original problem. 
In 1975, he received a Nobel prize in Economics; see 
Gangbo and McCann~\cite{gm-got-96} for mathematical and historical details.

While usually known as the {\em Monge-Kantorovich transportation problem},
the minimum cost of a transportation is sometimes called {\em Wasserstein metric}
or, in computer science, {\em earth mover's distance} between the two measures
$\mu$ and $\nu$. It can be used to measure similarity in image retrieval; 
see Rubner et al.~\cite{rtg-mdaid-98}.
If both measures $\mu$ and $\nu$ have finite support, Monge's problem becomes
the minimum weight matching problem for complete bipartite graphs, where edge weights
represent transportation cost; see Rote~\cite{r-tapm-09}, Vaidya~\cite{v-ghm-89} and Sharathkumar et al.~\cite{sa-atpgs-12}.

We are interested in the case where only measure $\nu$ has finite support. 
More precisely, we assume that a set $S$ of $n$ point sites $p_i$ is given, and
numbers $\lambda_i > 0$ whose sum equals~1. A body $C$ of volume~1 must be split into 
subsets $C_i$ of volume $\lambda_i$ in such a way that the total cost of transporting,
for each $i$, all points of $C_i$ to their site $p_i$ becomes a minimum.
In this setting, volume is measured by some measure $\mu$, and 
transport cost $d(z,z')$ by some measure of distance.

Gangbo and McCann~\cite{gm-got-96} report on the cases
where either $d(z,z') = h(z-z')$ with a strictly convex function $h$,
or $d(z,z') = l(|z-z'|)$ with a non-negative, strictly concave function $l$
of the Euclidean norm.
As a consequence of deep results on the general Monge-Kantorovich problem,  
they prove a surprising fact. The minimum cost partition of $C$ is given 
by the additively weighted Voronoi diagram of the sites $p_i$, based
on cost function $d$ and additive weights $w_i$. In this structure, the Voronoi
region of $p_i$ contains all points $z$ satisfying
\[
     d(p_i,z) - w_i \ < \  d(p_j,z) - w_j \mbox{ for all } j\not= i.
\]
Villani~\cite{v-oton-09} has more recently observed that this holds even for more general
cost functions, that need not be invariant under translations, and in the case where both distributions are continuous.

Figure~\ref{cones-fig} depicts how to obtain this structure in dimension $d=2$. For each point site $p\in S$
we construct in $\IR^{3}$ the cone $\{ (z,d(p,z)-w_p)\ | \ z \in \IR^2\}$.  The lower envelope of the cones, projected
onto the $XY$--plane, results in the Voronoi diagram.
\begin{figure}[hbtp]%
  \begin{center}%
    \includegraphics[scale=0.5,keepaspectratio]{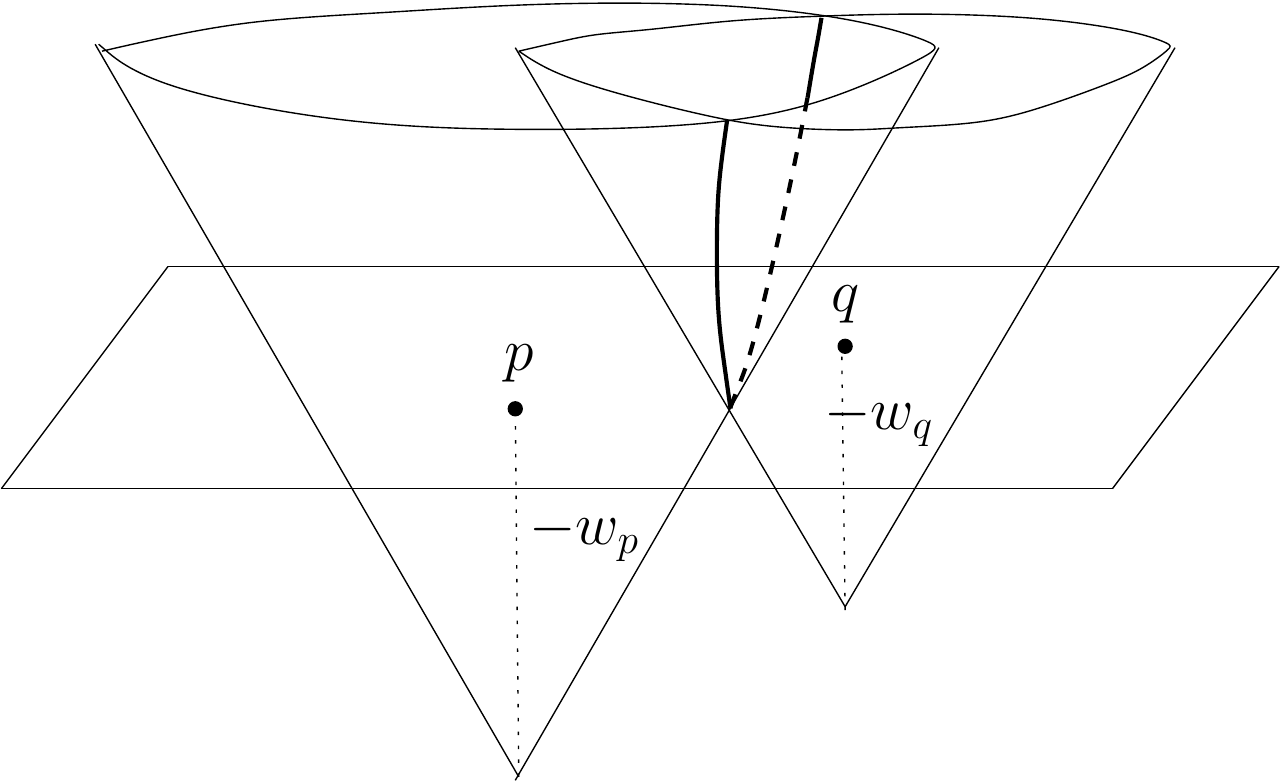}%
    \caption{An additively weighted Voronoi diagram as the lower envelope of cones.}%
    \label{cones-fig}
  \end{center}%
\end{figure}
Independently, Aurenhammer et al.~\cite{aha-mttls-98} have studied the case where $d(z,z')=|z-z'|^2$.
Here, a {\em power diagram} (see~\cite{ak-vd-99}) gives an optimum splitting of $C$. In addition
to structural results, they provide algorithms for computing the weights $w_i$.
Regarding the case where the transportation cost from point $z$ to site $p_i$ is given by individual cost functions $d_{p_i}(z)$, it is still unknown how to compute the weights $w_i$.

\vspace{\baselineskip}
In this paper we consider the situation where the cost $d(p_i,z)$ of transporting point $z$ to site $p_i$
is given by individual distance functions $d_{p_i}(z)$.
We require that the weighted Voronoi
diagram based on the functions $d_{p_i}(\cdot)$ is well-behaved, in the following sense.
Bisectors are $(d-1)$--dimensional, and increasing weight $w_i$ causes the bisectors of $p_i$ to sweep
$d$--space in a continuous way.
These requirements are fulfilled for the Euclidean metric and,
at least in dimension~2, if the distance function $d_p(z) = d(z-p)$ assigned to each site $p$ is a translate of the same strictly convex distance function $d(\cdot)$.

We show that these assumptions are strong enough to obtain the result of~\cite{gm-got-96,aha-mttls-98}:
The weighted Voronoi diagram based on the functions $d_{p_i}(\cdot)$ optimally solves the transportation problem,
if weights are suitably chosen.
Whereas~\cite{gm-got-96} derives this fact from a more general measure-theoretic theorem, our proof uses
the minimization of a quadratic objective function and
 geometric arguments.
The purpose of our paper is to show that such a simple proof is possible.

After stating some definitions in Section~\ref{defi-sec} we generalize arguments from~\cite{aha-mttls-98}
to prove, in Section~\ref{parti-sec}, that $C$ can be split into parts of arbitrary given
volumes by a weighted Voronoi diagram, for a suitable choice of weights.
Then, in Section~\ref{opti-sec},
we show that such a partition is optimal, and unique.
In Section~\ref{unique-sec} we show that if $C$ is connected, these weights are uniquely determined, up to addition of a constant.
%

\section{Definitions}     \label{defi-sec}

Let $\mu$ be a measure defined for all Lebesgue-measurable
subsets of $\IR^d$.
We assume that $\mu$ and the $d$-dimensional Lebesgue measure are mutually continuous, that is, $\mu$ vanishes exactly on the sets of Lebesgue measure zero.


Let $S$ denote a set of $n$ point sites in $\IR^d$. For each $p \in S$ we are given 
a continuous function 
\begin{eqnarray*}
d_p \colon \IR^d \to \IR_{\geq 0}
\end{eqnarray*}
that assigns to each point $z$ of $\IR^d$ a nonnegative value $d_p(z)$ as the ``distance''  
from site $p$ to $z$.

For $p \not=q \in S$ and $\gamma \in \IR$ we define
the \emph{bisector}
\[
     {B_\gamma}(p,q) := \{\, z \in \IR^d \mid d_p(z) - d_q(z) = \gamma\,\}
\] 
and the \emph{region}
\[
     {R_\gamma}(p,q) := \{\, z \in \IR^d \mid d_p(z) - d_q(z) < \gamma\,\}.
\]
The sets $R_{\gamma}(p,q)$ are open and increase with $\gamma$.
Now let us assume that for each site $p \in S$ an additive weight $w_p$ is given,
and let
\[
       w=(w_1, w_2, \ldots, w_n)
\]
denote the vector of all weights, according to some ordering $p_1, p_2, \ldots$  of $S$. Then,
$B_{w_i - w_j}(p_i,p_j)$ is called the {\em additively weighted bisector} of $p_i$ and $p_j$,  and
\[
     \mbox{VR}_w(p_i,S):= \bigcap_{j \not= i}  {R_{w_i -  w_j}}(p_i,p_j)
\]
is the {\em additively weighted Voronoi region} of $p_i$ with respect to $S$.
It consists of all points~$z$ for which $d_{p_i}(z) - w_i$ is smaller
than all values $d_{p_j}(z) - w_j$, by definition.
As usual, 
\[
        V_w(S) := \IR^d \setminus \bigcup_i \mbox{VR}_w(p_i,S)
\]
is called the {\em additively weighted Voronoi diagram} of $S$;
see~\cite{ak-vd-99}.  Clearly,  $\mbox{VR}_w(p_i,S)$ and $V_w(S)$ do not change if 
the same constant is added to all weights in $w$.
Increasing a single value $w_i$ will increase the size of $p_i$'s Voronoi region.

\begin{figure}[hbtp]%
  \begin{center}%
    \includegraphics[scale=0.42,keepaspectratio]{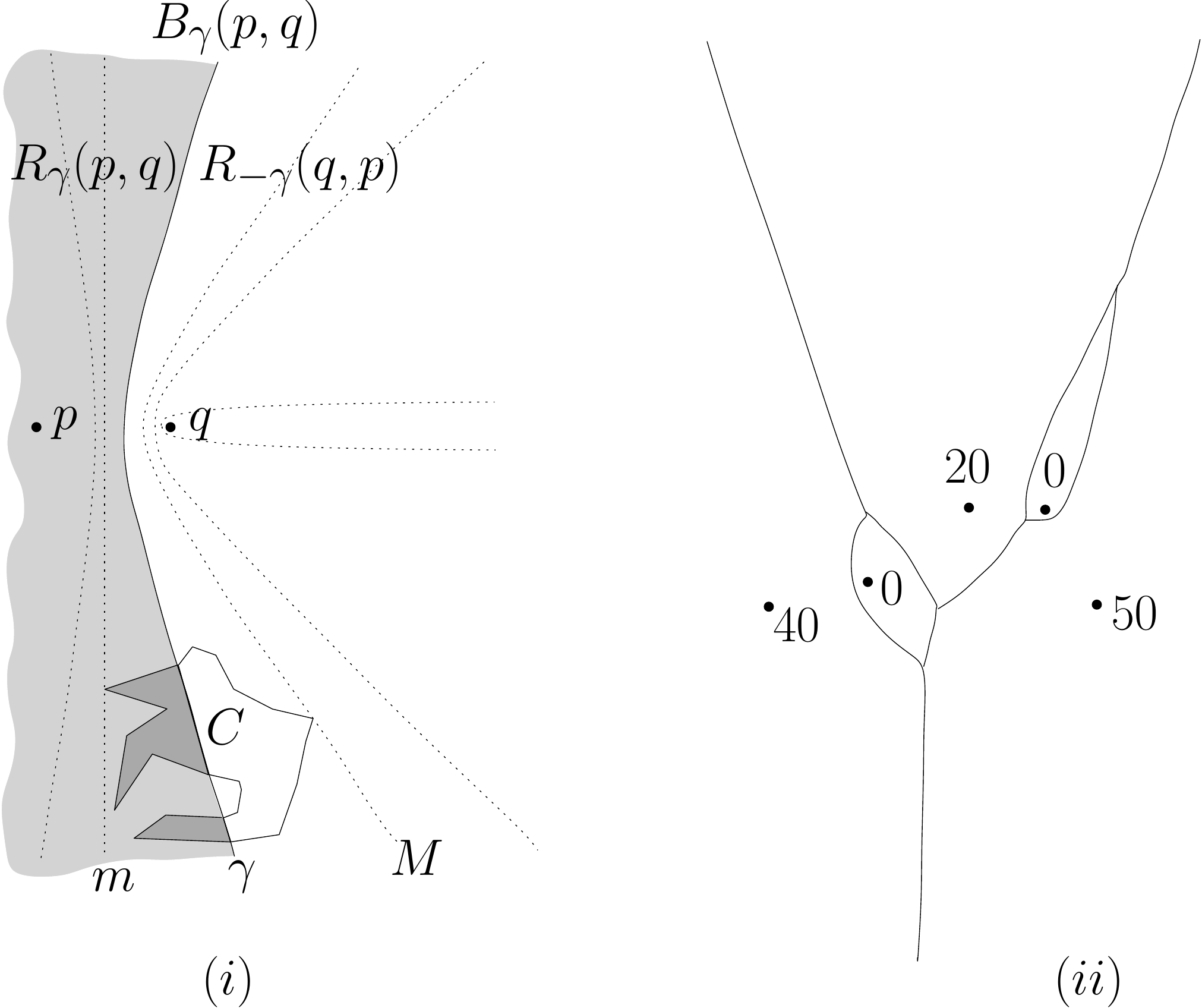}%
    \caption{(i) Sets $R_\gamma(p,q)$ for increasing values of $\gamma$.
                     (ii) An additively weighted Voronoi diagram.}%
    \label{hyperbolae-fig}
  \end{center}%
\end{figure}


\vspace{\baselineskip}
\noindent
{\bf Example.}   Let $d=2$ and $d_p(z)=|p-z|$ the Euclidean distance.
Given weights $w_p, w_q$, the bisector $B_{w_p - w_q}(p,q)$ is a line if $w_p=w_q$, 
and a branch of a hyperbola otherwise. Figure~\ref{hyperbolae-fig}(i) shows how 
$B_\gamma(p,q)$
 sweeps across the plane as $\gamma$ grows
from $-\infty$ to $\infty$. In fact, when $|\gamma|=|p-q|$, the bisector degenerates to a ray, and for larger  $|\gamma|$, the bisectors are empty.
The bisector $B_\gamma(p,q)$ forms the boundary of the set $R_\gamma(p,q)$, which increases with $\gamma$. 
Each bounded set $C$ in the plane will be reached at some point. 
From then on the volume of $C\cap R_\gamma (p,q)$ is continuously
growing until all of $C$ is contained in $R_\gamma(p,q)$.
Given $n$ points $p_j$ with additive weights $w_j$, raising the value of a single
weight $w_i$ will cause all sets $R_{w_i - w_j}(p_i, p_j)$ to grow until $C$ is fully
contained in the Voronoi region $V_w(p_i,S)$ of~$p_i$.

 Figure~\ref{hyperbolae-fig}(ii) shows an additively weighted Voronoi diagram $V_w(S)$ based on the Euclidean distance. 
It partitions the plane into~5 two-dimensional cells (Voronoi regions), and consists of~9 
cells of dimension~1 (Voronoi edges) 
 and of 5~cells
of dimension~0 (Voronoi vertices). Each cell is homeomorphic to an open sphere of appropriate dimension.

The next definition generalizes the above properties to the setting used in this paper.

\vspace{\baselineskip}
\begin{definition}       \label{admi-defi}
A system of continuous distance functions $d_p(\cdot)$, where $p \in S$, is called {\em admissible} if
for all $p \not= q \in S$, and for each bounded open set $C \subset \IR^d$,
there exist values $m_{pq}$ and $M_{pq}$ such that
$    \gamma \mapsto \mu \left(C \cap R_\gamma (p,q)\right)$
is continuously increasing from $0$ to $\mu(C)$ as 
$\gamma$ grows from $m_{pq}$ to $M_{pq}$.
Furthermore,
$C \cap R_\gamma (p,q) = \emptyset$ if $\gamma \leq m_{pq}$ and
$C \subset R_\gamma (p,q)$ if $M_{pq} \leq \gamma$.
\end{definition}

By symmetry we can assume w.~l.~o.~g.~$m_{qp} = - M_{pq}$ and $M_{qp} = - m_{pq}$;
see Figure~\ref{hyperbolae-fig}(i).
We will need the following structural property, which essentially says that bisectors have measure~0.
\begin{lemma}															\label{bisector-volume-lemm}
For a system of admissible distance functions $d_p(\cdot)$, where $p\in S$, for any two points $p\not=q \in S$ and any $\gamma \in \IR$, and 
for any bounded open set $C\subset \IR^d$, we have $\mu \left(C \cap B_\gamma (p,q)\right) = 0$.
\end{lemma}
\begin{proof}
 By definition of $R_\gamma(p,q)$ and $B_\gamma(p,q)$, for all $\varepsilon > 0$ and $\gamma\in \mathbb{R}$ the following inequality holds:
\[\mu(C\cap R_{\gamma + \varepsilon}(p,q)) \geq \mu(C\cap R_{\gamma}(p,q)) + \mu(C\cap B_{\gamma}(p,q))\]
We let $\varepsilon$ decrease to 0.
Since, according to Definition~\ref{admi-defi}, the function $\gamma \mapsto \mu \left(C \cap R_\gamma (p,q)\right)$ is continuous, we get
\[\lim_{\varepsilon \searrow 0} \left(\mu(C\cap R_{\gamma+\varepsilon}(p,q))\right) = \mu(C\cap R_{\gamma}(p,q)).\]
This implies $\mu(C\cap B_{\gamma}(p,q)) = 0$.
\end{proof}

\section{Partitions of prescribed size}     \label{parti-sec}
The following theorem shows that we can use an additively weighted Voronoi
diagram based on an admissible system of distance functions $d_p$ to partition $C$ into subsets of prescribed
sizes. 

\begin{theorem}          \label{partition-theo}
Let $n \geq 2$ and
let $d_{p_i}(\cdot)$,  $1 \leq i \leq n$,
be an admissible system as in 
Definition~\ref{admi-defi}.
 Let $C$ be a bounded open
subset of $\IR^d$. Suppose we are given $n$ real numbers 
$\lambda_i >0$ with
$\lambda_1 + \lambda_2 + \cdots + \lambda_n =\mu(C)$.
Then
there is a weight vector $w=(w_1, w_2, \ldots, w_n)$ such that
\[
      \mu(C \cap \mathrm{VR}_w(p_i,S)) = \lambda_i 
\]
holds for $1 \leq i \leq n$.

If, moreover, $C$ is pathwise connected
then $w$ is unique up to addition of a constant to all $w_i$, and
the parts $C_i = C \cap \mathrm{VR}_w(p_i,S)$ of this partition are unique.
\end{theorem}
\begin{proof}
The function
\[
     \Phi(w) := \sum_{i=1}^n \bigl( \mu(C \cap \mbox{VR}_w(p_i,S)) - \lambda_i   \bigr)^2
\]
measures how far $w$ is from fulfilling the theorem.
Since each function $\gamma \mapsto \mu(C \cap R_\gamma(p_i,p_j))$ is
continuous by Definition~\ref{admi-defi},
and because
\begin{multline*}
\lvert \mu(C \cap \mbox{VR}_w(p_i,S)) - \mu(C \cap \mbox{VR}_{w'}(p_i,S)) \rvert \leq\\
\sum_{j \not= i} \lvert \mu(C \cap R_{w_i-w_j}(p_i,p_j)) -  \mu(C \cap R_{w'_i-w'_j}(p_i,p_j))\rvert,
\end{multline*}
we conclude that the function $\Phi$ is continuous, too. 

Let 
$      D := \max \{\,M_{pq} \mid p \not= q \,\}$      
with $M_{pq}$ as in
Definition~\ref{admi-defi}.
Note that $m_{pq}=-M_{qp}$ and $m_{pq}<M_{pq}$ together imply that $D > 0$ is a positive constant.
On the compact set $[0,D]^n$, the function $\Phi$ attains its minimum value at some argument $w$.
If $\Phi(w) = 0$ we are done. 
Suppose that $\Phi(w) >0$. Since the volumes of the Voronoi regions inside $C$ add up to~$\mu(C)$,
there must be some sites $p_j$ whose Voronoi regions have too large an intersection with $C$ i.e., of volume $> \lambda_j$,
while other region's intersections with $C$ are too small.

We now show that by decreasing the weights of some points of $S$ we can decrease the value of $\Phi$. 
For any point $p_i\in S$ we consider the continuous {``excess'' function}
$$\phi_{p_i}(w) := \mu(C\cap \mbox{VR}_w(p_i,S)) - \lambda_i$$ that indicates by which amount of volume the region of $p_i$ is too large or too small.
We have $$\sum_{i=1}^n\phi_{p_i}(w)=0.$$
Let $T \subseteq S$ be the set of points $p_i$ for which $\phi_{p_i}(w)$ attains its maximum value, $\tau$.
By assumption, $T$ is neither empty nor equal to $S$.
We will reduce the weights of the points in $T$ by the same amount $\delta$, obtaining a new weight vector $w'$.
Note that, by construction, 
 $\mbox{VR}_w(t,S) \supseteq \mbox{VR}_{w'}(t,S)$
 for any site $t\in T$,
 and $\mbox{VR}_w(s,S) \subseteq \mbox{VR}_{w'}(s,S)$, 
 for any site  $s\in S \setminus T$.

Let $w'(\delta)$ denote the weight vector $(w'_1,\dots, w'_n)$ where $w'_i := w_i - \delta$ if $p_i\in T$, and $w'_i := w_i$ otherwise.
The volume of the set of points of $C$, that is assigned by $V_{w}(S)$ to some point in $T$ and by $V_{w'(\delta)}(S)$ to some point in $S\setminus T$ is \[\mbox{Loss}(\delta) := \sum_{t\in T}\bigl( \phi_t(w)-\phi_t(w'(\delta))\bigr)\]
Let $\tau := \max_{p\in S}\phi_{p}(w) > 0$ and $\tau' := \max_{s\in S\setminus T}\phi_s(w) < \tau$.
Our goal is to choose $\delta$ in such a way that
\begin{equation}
  \label{eq:loss}
  \text{Loss}(\delta) = \frac{\tau-\tau'}{2n}
\end{equation}
and to show that this choice decreases $\Phi$ strictly.
Since $\mbox{Loss}(0) = 0$ and $\mbox{Loss}(\delta)$ is a monotone and continuous function,
our aim is to solve \eqref{eq:loss} by the intermediate value theorem.
For all $\delta > 2D$ we observe that $\mbox{Loss}(\delta) = \sum_{p_i\in T}(\tau + \lambda_i) > \tau$ holds, since in this case the Voronoi region $\mbox{VR}_{w'(\delta)}(p_i,S)$ of every point $p_i\in T$ contains no point of $C$.
Now $0 < \frac{\tau-\tau'}{2n} < \tau$ implies that we can indeed solve \eqref{eq:loss} by the intermediate value theorem.
%
The latter inequality is evident if $\tau' \geq 0$. Otherwise, it follows from $-\tau'\leq
 -\sum_{s\in S \setminus T}\phi_s(w)
= \sum_{s\in  T}\phi_s(w)
 = |T|\cdot \tau \leq (n-1)\tau$.

In the following let $w' := w'(\delta)$, for short.
In order to prove the theorem we use the following reformulation of the statement $\Phi(w)>\Phi(w')$:
\begin{eqnarray}
\sum_{t\in T} \bigl(\phi_t(w)^2 - \phi_t(w')^2\bigr) > \sum_{s\in S\setminus T}\bigl( \phi_s(w')^2 - \phi_s(w)^2\bigr)\label{final-eqn}
\end{eqnarray}
Now, in order to
 prove 
 inequality~$(\ref{final-eqn})$, 
 we 
 give an upper bound on the right-hand side of the inequality,
which will be compared with a lower bound on the left-hand side of the inequality.
Let $\varepsilon_s \geq 0$ denote the volume increase, in $C$, of the region of site $s \in S \setminus T$.
By the choice of $\delta$, we have
 $\sum_{s\in S\setminus T}\varepsilon_s =\frac{\tau - \tau'}{2n}$, and
\begin{align*}
\sum_{s\in S\setminus T}\bigl(\phi_s(w')^2 - \phi_s(w)^2\bigr)
&= \sum_{s\in S\setminus T}\bigl((\phi_s(w)+\varepsilon_s)^2 - \phi_s(w)^2\bigr)\\
&= \sum_{s\in S\setminus T} 2\varepsilon_s\phi_s(w) + \sum_{s\in S\setminus T}\varepsilon_s^2\\
&\leq 2\tau'\sum_{s\in S\setminus T}\varepsilon_s + \Bigl(\sum_{s\in S\setminus T}\varepsilon_s\Bigr)^2\\
&= 2\tau'\frac{\tau - \tau'}{2n} + \left(\frac{\tau - \tau'}{2n}\right)^2
\end{align*}

We now give a lower bound on the left-hand side of~$(\ref{final-eqn})$.
If $\varepsilon_t \geq 0$ denotes the volume decrease, in $C$, of the region of site $t\in T$, then the corresponding calculation for the points in $T$ yields:
\begin{align*}
\sum_{t\in T}\bigl(\phi_t(w)^2 - \phi_t(w')^2\bigr)
&= \sum_{t\in T}\bigl(\phi_t(w)^2 - (\phi_t(w)-\varepsilon_t)^2\bigr)\\
&= \sum_{t\in T} 2\phi_t(w)\varepsilon_t - \sum_{t\in T}\varepsilon_t^2\\
&\geq 2\tau\sum_{t\in T}\varepsilon_t - \Bigl(\sum_{t\in T}\varepsilon_t\Bigr)^2\\
&=2\tau\frac{\tau-\tau'}{2n} - \left(\frac{\tau-\tau'}{2n}\right)^2
\end{align*}
where we have used the same arguments as above and also that $\phi_t(w) = \tau$, by definition, for all $t\in T$.
%
%
%
It is easy to verify the inequality
\begin{eqnarray}
2\tau\frac{\tau-\tau'}{2n} - \left(\frac{\tau-\tau'}{2n}\right)^2
> 2\tau'\frac{\tau - \tau'}{2n} + \left(\frac{\tau - \tau'}{2n}\right)^2\nonumber
\end{eqnarray}
for
$\tau-\tau' > 0$ and $n \geq 2$. 
Hence inequality~$(\ref{final-eqn})$ holds. 

We have now shown the existence of a weight vector $w'$ with $\Phi(w')<\Phi(w)$.
%
%
If $w'$ belongs to $[0,D]^n$, we are done. Otherwise,
%
%
suppose the maximum weight $w'_m$ occurs for some point $p'_m$.
We add $D-w'_m$ to all weights. This does not change the Voronoi diagram defined by $w'$, but now the maximum weight $w'_m$ is exactly $D$. 
Next we observe that if now the weight of some point $p_j$ is $<0$, then its Voronoi region is empty, because $C$ is contained in the set $R_{w'_m-w'_j}(p_m,p_j)$, by definition of $D$. Assigning weight $0$ to each of those points also results in empty Voronoi regions for those points. Altogether, we now have $0 \leq w'_i \leq D$ for all $p_i\in S$. 

We have shown the existence of a weight vector $w'\in [0,D]^n$ that satisfies $\Phi(w')<\Phi(w)$.
This contradicts the minimality of $\Phi(w)$.
Hence the minimum value of $\Phi$ is $0$.

Uniqueness of the partition 
 will be discussed in Section~\ref{unique-sec}.
\end{proof}

\section{Optimality}     \label{opti-sec}
Again, let $d_{p}(\cdot)$,  $p \in S = \{p_1, p_2, \ldots, p_n\}$,
be an admissible system as in Definition~\ref{admi-defi}, 
and let $C$ denote a bounded and open subset of $\IR^d$. 
Moreover, we are given real numbers $\lambda_i >0$ whose sum equals~$\mu(C)$.

By Theorem~\ref{partition-theo} there
exists a weight vector $w=(w_1, w_2, \ldots, w_n)$ satisfying
\[
      \mu(C \cap \mbox{VR}_w(p_i,S)) = \lambda_i  \mbox{ for } i=1, \ldots, n. 
\]
Now we prove that this subdivision of $C$
minimizes the transportation cost, i.~e., the average distance from each 
point to its site.
It is convenient to describe partitions of $C$ by $\mu$-measurable maps $f\colon C\to S$
where, for each $p \in S$, $f^{-1}(p)$ denotes the region assigned to $p$.
Let $F_\Lambda$ denote the set of those maps $f$ satisfying
$\mu(f^{-1}(p_i)) = \lambda_i 
$ for $i = 1, \ldots, n$.

\begin{theorem}                       \label{opti-theo}
The partition of $C$ into regions $C_i := C \cap \mathrm{VR}_w(p_i,S)$
minimizes
\[\mathrm{cost}(f) := \int_C \! d_{f(z)}(z) \, \mathrm{d} \mu(z)\]
over all maps $f \in F_\Lambda$. Any other partition of minimal cost differs at most by
sets of measure zero from $(C_i)_i$.
\end{theorem}

If the measure $\mu$ were a discrete measure concentrated on a finite
set $C$ of points, the optimum partition problem would turn into the
classical transportation problem, a special type of network flow
problem on a bipartite graph. The weights $w_i$ would be obtained as
dual variables. 
The following proof mimics the
optimality proof for the discrete case.

\begin{proof}

Let $f_w$ be defined by $f_w (C_i)=\{p_i\}$ for all $i$, and $f\in F_\Lambda$ be another map.
If we define $c(z,p) := d_p(z)$, then the cost of map $f$ is
\begin{align}
\text{cost}(f)
&= \int_C \left(c(z,f(z)) - w_{f(z)} + w_{f(z)}\right) \, \mathrm{d} \mu(z)
\nonumber \\
&= \int_C \left(c(z,f(z)) - w_{f(z)}\right) \, \mathrm{d} \mu(z) + \int_C w_{f(z)} \, \mathrm{d} \mu(z)\label{cost-estimate-line-one-eqn}\\
&\geq \int_C \left(c(z,f_w(z)) - w_{f_w(z)}\right) \, \mathrm{d} \mu(z) + \int_C w_{f(z)} \, \mathrm{d} \mu(z)\label{cost-estimate-line-two-eqn}\\
&= \text{cost}(f_w)
 - \int_Cw_{f_w(z)} \, \mathrm{d} \mu(z) + \int_C w_{f(z)} \, \mathrm{d} \mu(z)\label{cost-last-line-eqn}
\end{align}
where the inequality between lines~$(\ref{cost-estimate-line-one-eqn})$ and~$(\ref{cost-estimate-line-two-eqn})$ follows from the fact that
\begin{equation}
  \label{eq:weight-better}
c(z,f_w(z)) - w_{f_w(z)} \leq c(z,f(z)) - w_{f(z)},
\end{equation}
 by the definition of additively weighted Voronoi diagrams. Here we have assumed for simplicity that the map $f_w$ assigns every point $p\in V_w(S)$ to \emph{some} weighted nearest neighbor of $p$ in $S$, according to some tie-break rule. Since Lemma~\ref{bisector-volume-lemm} implies $\mu(C \cap V_w(S)) = 0$, changing the assignment of those points has no influence on $\mbox{cost}(f_w)$.

Both maps $f$ and $f_w$ partition $C$ into regions of volume $\lambda_i$, $1\leq i\leq n$, i.~e., $\mu(f_w^{-1}(p_i)) = \mu(f^{-1}(p_i)) = \lambda_i$ for all $p_i\in S$. Therefore, 
\begin{equation*}
\int_C w_{f(z)} \, \mathrm{d} \mu(z)
= \sum_{p_i\in S}\int_{f^{-1}(p_i)}w_{f(z)} \, \mathrm{d} \mu(z)
= \sum_{p_i\in S} \lambda_i w_i,
\end{equation*}
and the same value is obtained for $f_w$:
\begin{equation*}
\int_C w_{f_w(z)} \, \mathrm{d} \mu(z)
= \sum_{p_i\in S}\int_{f_w^{-1}(p_i)}w_{f_w(z)} \, \mathrm{d} \mu(z)
= \sum_{p_i\in S} \lambda_i w_i.
\end{equation*}
%
%
%
Now the optimality claim follows from~$(\ref{cost-last-line-eqn})$.

We still have to prove uniqueness of the solution, up to a set of measure~0.
Assume that $f$ differs from $f_w$ on some set $G$ of positive measure.
We are done if we show that the inequality between~$(\ref{cost-estimate-line-one-eqn})$ and~$(\ref{cost-estimate-line-two-eqn})$
holds as a strict inequality.
We can remove the points of $V_w(S)$ from $G$, since they have measure 0.
Then \eqref{eq:weight-better} holds as a strict inequality on $G$.
The difference
between~$(\ref{cost-estimate-line-one-eqn})$ and~$(\ref{cost-estimate-line-two-eqn})$ is then the integral of the positive function
$g(z) := c(z,f_w(z)) - w_{f_w(z)} - \bigl( c(z,f(z)) - w_{f(z)}\bigr)$ over a
set $G$ of positive measure. Such an integral 
$\int_{z\in G} g(z)\,\mathrm{d}\mu(z)$
has always a positive value:
the countably many sets
 $\{\,z\in G \mid \frac1k >g(z)\ge \frac1{k+1}\,\}$
and
 $\{\,z\in G \mid g(z)\ge 1\,\}$
partition $G$, and thus their measures add up to $\mu(G)$. Therefore,
at least one of these sets has positive measure, and it follows directly
that the integral is positive.
\end{proof}

\section{Uniqueness}\label{unique-sec}

In contradistinction to Theorem~\ref{opti-theo}, the weight vector $w$
and the resulting weighted Voronoi partition
 in Theorem~\ref{partition-theo} is not unique:
if $C$ is disconnected, changing $w$ may move the
bisectors between different connected components of $C$ without
affecting the partition, or only changing it in boundary points of $C$.



However, uniqueness can be obtained under the additional assumption that
 $C$ is pathwise connected. 
 Then
 in Theorem~\ref{partition-theo},
the weight vector $w=(w_1,\dots,w_n)$ that partitions $C$ into regions of prescribed size is unique up to addition of the same constant to all $w_i$.
Consequently, the parts in the
weighted Voronoi partition
outside of $C$
are also uniquely determined.

The proof uses the following lemma.

\begin{lemma}\label{decrease-lemm}
  Assume that $C$ is path-connected, in addition to being an open
  bounded set.  Consider a partition of the point set $S$ in two
  nonempty sets $T$ and $S\setminus T$.  Suppose, that, for a given
  weight vector $w$, all regions $C\cap \mbox{VR}_w(p,S)$ have
  positive measure.

  If we increase the weight of every point in $T$ by the same constant
  $\delta>0$, then $\mu(C\cap \mbox{VR}_w(t,S))$ increases, for some
  $t\in T$.
\end{lemma}

\begin{proof}
In this section, we will omit the reference to the set $S$ from the Voronoi regions and simply write
$ \mbox{VR}_w(p)$ instead of
$ \mbox{VR}_w(p,S)$, since the point set $S$ is fixed.

It is clear that the regions $\mbox{VR}_w(t)$
for $t\in T$ can only
grow, in the sense of gaining new points, but we have to show that
this growth results in a strict increase of~$\mu(C\cap\mbox{VR}_w(t))$
for some $t\in T$.

Consider the continuous function
$$d_T(z) := \min \{\,d_p(z)-w_p\mid p\in T\,\}$$
and the function
$d_{S\setminus T}(z)$, which is defined analogously,
and define 
\begin{align*}
\mathrm{VR}_w(T) & := \{\,z\in \IR^d \mid d_T(z) < d_{S\setminus T}(z)\,\}
\\
\mathrm{VR}_w(S\setminus T) & := \{\,z\in \IR^d \mid d_T(z) > d_{S\setminus T}(z)\,\}
\end{align*}
These two sets partition $\IR^d$, apart from some ``generalized bisector''
 where equality holds.

Roughly,
$\mathrm{VR}_w(T)$ is obtained by merging
 all Voronoi regions $ \mbox{VR}_w(t)$ for $t\in T$ into one set,
together with some bisecting boundaries between them.
Since these boundaries have zero measure inside $C$, we have
\begin{equation}
  \label{eq:sum-parts}
 \mu(C\cap \mbox{VR}_w(T)) 
= \sum_{t\in T} \mu(C\cap \mbox{VR}_w(t))
\end{equation}
Let $w = (w_1,\dots, w_n)$ and $w' = (w'_1,\dots, w'_n)$, where $w'_i
:= w_i+\delta$ if $s_i\in T$ and $w'_i := w_i$ otherwise.
It suffices to show that in going from $w$ to $w'$,
$ \mu(C\cap \mbox{VR}_w(T)) $ increases, since then, by
\eqref{eq:sum-parts}, one of the constituent Voronoi regions
$\mu(C\cap \mbox{VR}_w(t))$ must increase.
As mentioned, the set
$C\cap\mbox{VR}_w(T)$ can only grow, but we have to show that this
growth results in a strict increase of~$\mu$.

If $C\cap\mathrm{VR}_{w'}(S\setminus T)=\emptyset$, we are done because
$\mu(C\cap \mbox{VR}_{w'}(T))$ has increased to $\mu(C)$.
Otherwise, consider a path 
in $C$ connecting some point $p_1 \in
\mbox{VR}_w(T)$ with some point $p_2 \in \mathrm{VR}_{w'}(S\setminus
T)$.
We have, by definition,
$d_T(p_1) < d_{S\setminus T}(p_1)$ and
$d_T(p_2)-\delta > d_{S\setminus T}(p_2)$.
Hence, by the intermediate value theorem, we can find a point $p$
on the path 
 with
$d_T(p)-\delta/2= d_{S\setminus T}(p)$.
This point is an interior point of
$C\cap \mbox{VR}_{w}(S\setminus T)$
and $C\cap \mbox{VR}_{w'}(T)$, and hence there is a neighborhood of $p$
which adds a positive measure to
$C\cap \mbox{VR}_{w}(T)$ when going from $w$ to $w'$.
\end{proof}

To finish the proof of uniqueness in Theorem~\ref{partition-theo},
suppose that two weight vectors $w, w'$ do the job of producing the
desired measures $\lambda_i$ for the regions, but $w' - w \not= (c, c,
\ldots, c)$ for any constant $c$.
 We may assume,  by adding a constant to all entries, that $w\le w'$, and $w_i =w'_i$ for some $i$.
We will now gradually change $w$ into $w'$ by modifying its entries.
Let $T := \{\,s_i\in S \mid w'_i-w_i = \max_j (w'_j-w_j)\,\}$.
 By assumption, $T$ is neither empty nor equal to~$S$.
We increase the weights of the points in $T$, leaving
the remaining values fixed. The amount $\delta$ is chosen in such
a way that
$w'_i-w_i$ for $i\in T$ does not decrease below the remaining
differences $w'_j-w_j$ for $j\in S\setminus T$.
By Lemma~\ref{decrease-lemm}, the measure of the region assigned to
some point $t_0\in T$ strictly increases in this process.  When the
limiting value $\delta$ is reached, the set $T$ of points where
$w'_i-w_i$ achieves its maximum is enlarged, and the whole process is
repeated until $w=w'$. During this process, $t_0$ will always be among
the points whose weight is increased, and therefore, the measure of
its region can never shrink back to the original value---a contradiction.
\qed


\section{Conclusion and future work}

We have given a geometric proof for the fact that additively weighted Voronoi diagrams
can optimally solve some cases of the Monge-Kantorovich transportation problem, where one measure has
finite support. Surprisingly, the existence of an optimal solution---the main mathematical challenge in the general
case---is an easy consequence of our proof. In remains to be seen to which extent our assumptions on the
distance functions $d_{p_i}$ can be further generalized.

\bibliographystyle{model1-num-names}







\end{document}